\documentclass[10pt,leqno,a4paper]{article}
\usepackage{xcolor}
\usepackage{amsfonts}
\usepackage{amsmath, amsfonts, amsthm, amssymb, amscd}
\usepackage[mathcal]{euscript}
\usepackage{mathrsfs}
\usepackage{dsfont}
\usepackage{ulem}

\usepackage{CJK}

\newtheorem{theorem}{Theorem}[section]
\newtheorem{proposition}[theorem]{Proposition}
\newtheorem{lemma}[theorem]{Lemma}

\newtheorem{definition}[theorem]{Definition}
\newtheorem{remark}[theorem]{Remark}

\numberwithin{equation}{section}

\usepackage{a4wide}
\usepackage{multirow}
\usepackage{tabularx}
\usepackage{lscape}

\newcommand \Kcal {\mathcal K}
\newcommand \Hcal {\mathcal H}

\newcommand \Ecal {\mathcal {E}}

\newcommand \xb {\bar {x}}
\newcommand \delb {\bar {\del}}

\newcommand \Tb {\overline {T}}
\newcommand \Phib{\overline{\Phi}}
\newcommand \Psib{\overline{\Psi}}

\newcommand \mb{\overline m}
\newcommand \Qb{\overline Q}

\newcommand \del \partial
\newcommand \delu {\underline{\del}}
\newcommand \Tu {\underline{T}}

\newcommand \minu{\underline{m}}

\newcommand \Psiu{\underline{\Psi}}
\newcommand \Phiu{\underline{\Phi}}
\newcommand \Qu{\underline{Q}}

\newcommand \RR{\mathbb{R}}

\newcommand {\seq}{\overset{*}{=}}

\makeatletter
\def\hlinew#1{%
  \noalign{\ifnum0=`}\fi\hrule \@height #1 \futurelet
   \reserved@a\@xhline}
\makeatother

\let\oldmarginpar\marginpar
\renewcommand\marginpar[1]{\-\oldmarginpar[\raggedleft\footnotesize #1]%
{\raggedright\footnotesize #1}}

\begin{document}
\title{Remarks on basic calculus in hyperboloidal foliation}
\author{Yue MA}

\maketitle

\section{Introduction}
In this article we will give a systematic summary on the basic calculus in the hyperboloidal foliation context, especially the estimates based on commutators. These results are frequently applied in many context ( see in detail \cite{LeFloch-MA-book-2015, LM2, LM3, MA2017-1, MA2017-2, MA2017-3}), however, they have not been systematically stated and proved.

This article is devoted to the discussion of the following aspects: the basic definition on hyperboloidal foliation, the vector fields, the notion and calculus with multi-indices, the estimates based on null condition, the decomposition of commutators, the estimates on Hessian form.  All of the above results will be restated and proved in detail. The notation will also be reorganized in order to simplify the calculation. Some results such as calculus with multi-indices, the generalized Leibniz rule and the generalized F\`aa di Bruno's formula, which have been applied in an implicit manner in many occasion, will be explicitly stated and proved for the first time.

\section{Basic notation}

\subsection{Frames and vector fields}
We are working in $\RR^{n+1}$ ($n\in\mathbb{N}, n\geq 1$) equipped with Minkowski metric\footnote{We make the signature convention $(+,-,\cdots,-)$. However most of the results are irrelevant with is choice of signature.}. The canonical coordinates are denoted by $(t,x) = (t,x^1,x^2\cdots, x^n) = (x^0,x^1,x^2\cdots, x^n)$.

Throughout this article we make the following convention on index: a Greek letter represents an index contained in $\{0,1,\cdots n\}$ while a Latin letter denotes an index contained in $\{1,2,\cdots n\}$. For example, $x^{\alpha}$ may refer to $t$ or $x^1$ but $x^a$ refers to one of $x^1,x^2,\cdots x^n$.

For $x\in\RR^n$, $|x|$ denotes the Euclidean norm of $x$. In some case we also denote by $r = |x|$. Furthermore, let $\Kcal = \{(t,x)|t>|x|+1\}$ be the interior of the (translated) light-cone and $\del\Kcal = \{(t,x)|t=|x|+1\}$ be its boundary.

For $s>0$, we denote by $\Hcal_s = \{(t,x)|t = \sqrt{s^2+|x|^2}\}$ the upper sheet of the hyperboloid with (hyperbolic) radius $s$. Then the inner part of the hyperboloid is defined as following:
$$
\Hcal_s^* = \Hcal_s\cap \Kcal.
$$
Then $\Kcal$ is foliated as
$$
\Kcal = \bigcup_{1<s<\infty} \Hcal_s^*.
$$
For $1<s_0<s_1$, we introduce
$$
\Kcal_{[s_0,s_1]} := \{(t,x)\in\Kcal|s_0^2+|x|^2\leq t^2\leq s_1^2+|x|^2\}
$$
the part of $\Kcal$ limited by $\Kcal_{s_0}$ and $\Kcal_{s_1}$. Similarly:
$$
\Kcal_{[s_0,\infty)} := \{(t,x)\in\Kcal|s_0^2+|x|^2\leq t^2\}.
$$

Now in $\Kcal$, we introduce the following vector fields:
$$
\delu_a: = \frac{x^a}{t}\del_t + \del_a
$$
which are tangent to $\Hcal_s$ with $s = \sqrt{t^2-|x|^2}$. Also, we denote by
$$
\delu_0 = \del_t.
$$
Then in $\Kcal$, $\{\delu_{\alpha}\}$ forms a frame, called the {\sl semi-hyperbolidal frame} (SHF for short). The transition matrices between this frame and the canonical frame read as:
$$
\big(\Phiu_\alpha^{\beta}\big)
= \big({\Phiu_\alpha}^{\beta}\big)
=\left(
\begin{array}{cc}
1 &0
\\
x/t &\text{Id}
\end{array}
\right),
\qquad
\qquad
\big(\Psiu_\alpha^{\beta}\big)
= \big({\Psiu_\alpha}^{\beta}\big)
=\left(
\begin{array}{cc}
1 &0
\\
-x/t &\text{Id}
\end{array}
\right).
$$
with $\delu_\alpha = \Phiu_\alpha^\beta \del_\beta$ and $\del_\alpha = \Psiu_\alpha^\beta \delu_\beta$. In the above expression $\text{Id}$ is the $n\times n$ order identity matrix.

In $\Kcal$, we introduce the {\sl hyperbolic variables}
$$
s := \sqrt{t^2-x^2},\quad  \xb^a = x^a.
$$
The canonical frame of these variables reads as:
$$
\del_{\xb^a}:=\delb_a = \delu_a ,\quad \del_s := \delb_0 = (s/t)\del_t.
$$
This frame is called the {\sl hyperbolic frame} (HF for short). The transition matrices between this frame and the canonical frame read as
$$
\big(\Phib_\alpha^{\beta}\big)
= \big({\Phib_\alpha}^{\beta}\big)
=\left(
\begin{array}{cc}
s/t &0
\\
x/t &\text{Id}
\end{array}
\right),
\qquad
\qquad
\big(\Psib_\alpha^{\beta}\big)
= \big({\Psib_\alpha}^{\beta}\big)
=\left(
\begin{array}{cc}
t/s &0
\\
-x/s &\text{Id}
\end{array}
\right).
$$

A tensor defined in $\Kcal$ can be expressed with respect to different frames. For example, let $T$ be a two-contravariant tensor field, then
$$
T = T^{\alpha\beta}\del_{\alpha}\otimes\del_{\beta} = \Tu^{\alpha\beta}\delu_{\alpha}\otimes\delu_{\beta} = \Tb^{\alpha\beta}\delb_{\alpha}\otimes\delb_{\beta}
$$
where $T^{\alpha\beta},\Tu^{\alpha\beta}$ and $\Tb^{\alpha\beta}$ are components of $T$ expressed in different frames respectively\footnote{Here and in the following discussion, we take the Einstein's convention of summation. Furthermore, when taking a sum with respect to a Greek index, the sum is taken over $\{0,1,\cdots n\}$ while summing with a Latin index, the sum is taken over $\{1,2,\cdots n\}$.}. The transition relations between these components are:
\begin{equation}\label{eq 1 transition}
\Tu^{\alpha\beta} = \Psiu^{\alpha}_{\alpha'}\Psiu^{\beta}_{\beta'}T^{\alpha'\beta'},\quad
\Tb^{\alpha\beta} = \Psib^{\alpha}_{\alpha'}\Psib^{\beta}_{\beta'}T^{\alpha'\beta'}.
\end{equation}

In a special case, the Minkowski metric (contravariant form) $m = m^{\alpha\beta}\del_{\alpha}\otimes\del_{\beta}$ is expressed in SHF and HF as following:
$$
\left(\minu^{\alpha\beta}\right)_{0\leq \alpha,\beta\leq n} =
\left(
\begin{array}{cc}
(s/t)^2 &x/t
\\
x/t &\text{Id}
\end{array}
\right),
\quad \quad
\left(\mb^{\alpha\beta}\right)_{0\leq \alpha,\beta\leq n}
= \left(
\begin{array}{cc}
1 &x/s
\\
x/s &\text{Id}
\end{array}
\right).
$$

\subsection{Families of vector fields and multi-index}
For the convenience of discussion, we introduce the following notation on families of vector fields:
\begin{itemize}
\item Lorentzian boosts, denoted by $\mathscr{L} = \{L_a|a=1,2,\cdots n\}$ with $L_a := x^a\del_t + t\del_a$.
\\
\item Partial derivatives, denoted by $\mathscr{P} = \{\del_{\alpha}|\alpha=0,1,\cdots n\}$.
\\
\item Adapted partial derivatives, denoted by $\mathscr{A} = \{(s/t)\del_{\alpha}|\alpha = 0,1,\cdots n\}$.
\\
\item Hyperbolic derivatives, denoted by $\mathscr{H} = \{\delu_a|a=1,2,\cdots n\}$.
\end{itemize}
We denote by
$$
\mathscr{Z} = \mathscr{L} \cup \mathscr{P} \cup \mathscr{A} \cup \mathscr{H}
$$
and
$$
Z_i =
\left\{
\aligned
&L_i,  &&i=1,2,\cdots, n,
\\
&\del_{i-(n+1)},  &&i=n+1,\cdots, 2n+1,
\\
&(s/t)\del_{i-2(n+1)}, &&i=2n+2,\cdots, 3n+2,
\\
&\delu_{i-(3n+2)},&&i=3n+3,\cdots ,4n+2.
\endaligned
\right.
$$
Then we introduce the following notation on high-order derivatives. Let $I = (i_1,i_2,\cdots i_N)$ be a multi-index with $i_j\in\{1,2,\cdots, 4n+2\}$ and $|I|=N$. Then
$$
Z^I : = Z_{i_1}Z_{i_2}\cdots Z_{i_N}
$$
is an $N-$orde differential operator.

Suppose that $Z^I$ is composed by $j$ Lorentzian boots, $i$ partial derivatives, $k$ adapted partial derivatives and $l$ hyperbolic derivatives, then $Z^I$ is said to be of type $(j,i,k,l)$. If $Z^I$ is of type $(j,0,0,0)$, we denote by $Z^I = L^I$ and if $Z^I$ is of type $(0,i,0,0)$, we denote by $Z^I = \del^I$. We use the notation $\del^IL^J$ for the following operator:
$$
\del^IL^Ju:= \del^I\left(L^Ju\right).
$$

Then we introduce the notion of partition of a multi-index. Let $I = (i_1,i_2,\cdots ,i_N)$ be an $N-$order multi-index. The following set associated to $I$
$$
\mathcal{G}(I) := \{(k,i_k)|k=1,2,\cdots N\}
$$
is called the {\bf graph} of $I$.

For a set of the following form:
\begin{equation}\label{eq 1 partition}
\mathcal{A} = \left\{(k,i_k)|k\in A\  \text{ finite subset of }\  \mathbb{Z}, i_k\in\{1,2,\cdots 4n+2\}\right\},
\end{equation}
$A$ is called the {\bf domain} of $\mathcal{A}$, and denoted by
$$
A = \text{D}(\mathcal{A}).
$$
Clearly, $D(\mathcal{G}(I)) = \{1,2,,\cdots, N\}$.

An {\bf $m-$partition of $I$} consists of an $m$-tuple $(\mathcal{I}_1,\mathcal{I}_2,\cdots \mathcal{I}_m)$ where $\mathcal{I}_j$ are subsets of $\mathcal{G}(I)$, such that
$$
\mathcal{I}_i\cap\mathcal{I}_j = \emptyset,\quad \mathcal{I}=\bigcup_{k=1}^{m}\mathcal{I}_k .
$$
We denote by $\mathscr{D}_m(I)$ the set of all its $m-$partitions.

An {\bf $m^*-$partition of $I$} consists of an $m$-element set $\{\mathcal{I}_1,\mathcal{I}_2,\cdots,\mathcal{I}_m\}$ where $\mathcal{I}_j$ are {\sl non-empty} subsets of $\mathcal{G}(I)$, such that
$$
\mathcal{I}_i\cap\mathcal{I}_j = \emptyset,\quad \mathcal{I}=\bigcup_{k=1}^{m}\mathcal{I}_k .
$$
We denote by $\mathscr{D}^*_m(I)$ the set of all its $m^*-$partition. Let $\{\mathcal{I}_1,\mathcal{I}_2,\cdots,\mathcal{I}_m\}\in\mathscr{D}_m^*(I)$. We define a total order on this set by
$$
\mathcal{I}_j\prec\mathcal{I}_k,\quad \text{if}\quad \min D(\mathcal{I}_j) < \min D(\mathcal{I}_k).
$$
Remark that $D(\mathcal{I}_j)$, being the domain of $\mathcal{I}_j$, is a finite subset of $\mathbb{Z}$ and $\min D(\mathcal{I}_j)$ always exists. Furthermore, because $D(\mathcal{I}_j)\cap D(\mathcal{I}_k) = \emptyset$, $\min D(\mathcal{I}_j) \neq \min D(\mathcal{I}_k)$ if $j\neq k$. From now on, when we write  $\{\mathcal{I}_1,\mathcal{I}_2,\cdots,\mathcal{I}_m\}\in\mathscr{D}_m^*(I)$, we always suppose that
$$
\mathcal{I}_j\prec\mathcal{I}_k,\quad\text{if} \quad j<k.
$$

Let $\mathcal{A}$ be a set of the form in \eqref{eq 1 partition}, the following multi-index $Z^J=Z_{j_1}Z_{j_2}\cdots Z_{j_l}\cdots$ is called the {\bf realization} of $\mathcal{A}$ with
$$
Z_{j_l} =
\left\{
\aligned
&Z_{i_k},\quad &&\text{ if } l = k\in \text{D}(\mathcal{A}),
\\
&1,&&\text{ otherwise}.
\endaligned
\right.
$$
It is clear that if $\mathcal{A}=\emptyset$, then its realization is $1$. Also, $A$ is a finite set, so there are only finite many factors $Z_{j_k}$ of $Z^J$ different from $1$.  Thus $Z^J$ is a finite order differential operator and its order equals to the size of $\text{D}(\mathcal{A})$. Remark that for a multi-index $I$, the realization of $\mathcal{G}(I)$ is $Z^I$.

An {\bf $m-$admissible decomposition of $Z^I$} is an $m-$tuple of operators $(Z^{I_1},Z^{I_2},\cdots Z^{I_m})$ with $2\leq m$ such that:
there is an $m-$partition of $\mathcal{G}(I)$, denoted by $(\mathcal{I}_1,\mathcal{I}_2\cdots ,\mathcal{I}_m)$, such that $Z^{I_j}$ are the realization of $\mathcal{I}_j$ with $j=1,2,\cdots m$. Remark that for a given $I$, two different $m-$partitions of $I$ may give the same $m-$admissible decomposition of $Z^I$. Here is an example:
$$
Z^I(u_1\cdot u_2) = \del_{\alpha}^2(u_1\cdot u_2) = \del_{\alpha}^2u_1\cdot u_2 + u_1\cdot \del_{\alpha}^2u_2 +  \del_{\alpha}u_1\cdot \del_{\alpha}u_2 + \del_{\alpha}u_1\cdot \del_{\alpha}u_2.
$$
These four terms correspond to the following four $2-$decomposition of $I$:
$$
\big(\{(1,\alpha),(2,\alpha)\},\emptyset\big),\quad \big(\emptyset,\{(1,\alpha),(2,\alpha)\}\big),\quad \big(\{(1,\alpha)\},\{(2,\alpha)\}\big),
\quad \big(\{(2,\alpha)\},\{(1,\alpha)\}\big).
$$

In the same manner, an {\bf $m^*-$admissible decomposition of $Z^I$} is an $m-$element set of operators $\{Z^{I_1},Z^{I_2},\cdots Z^{I_m}\}$ with $2\leq m$ such that:
there is an $m^*-$partition of $\mathcal{G}(I)$, denoted by $\{\mathcal{I}_1,\mathcal{I}_2\cdots ,\mathcal{I}_m\}$, such that $Z^{I_j}$ are the realization of $\mathcal{I}_j$ with $j=1,2,\cdots m$.

In the following discussion, we denote by $I_1+I_2+\cdots +I_m = I$ an $m$-admissible decomposition of $I$ and by $I_1+I_2+\cdots +I_m \seq I$ an $m^*-$addmisisble decomposition of $I$. It is obvious that
$$
|I| = |I_1|+|I_2|\cdots +|I_m|.
$$
Furthermore, suppose that $I_n$ is of type $(j_n,i_n,k_n,l_n)$ and $I$ is of type $(j,i,k,l)$. Then
$$
(j,i,k,l) = \sum_{n=1}^m(j_n,i_n,k_n,l_n).
$$

Quite often we apply the following notation
$$
\sum_{I_1+I_2+\dots +I_m=I}Z^{I_1}u_1\cdot Z^{I_2}u_2\cdots Z^{I_m}u_m
$$
for a sum over $\mathscr{D}_m(I)$ and
$$
\sum_{I_1+I_2+\dots +I_m\seq I}Z^{I_1}u\cdot Z^{I_2}u\cdots Z^{I_m}u
$$
for a sum over $\mathscr{D}^*_m(I)$. More precisely, for each $m-(\text{ or }m^*-)$partition of $I$ there is one term being its realization in the sum.

Now we are ready to state the following Leibniz rules and generalized F\`aa di Bruno's formula. Their proofs are detailed in Appendix \ref{sec A-1}.
\begin{lemma}[Leibniz Rules]\label{lem 1 Leinbiz}
If $u_k$ are functions defined in $\Kcal$, sufficiently regular, then
\begin{equation}\label{eq 1 Leibniz}
Z^I\big(u_1\cdot u_2\cdots u_m\big) = \sum_{I_1+I_2+\cdots I_m=I}Z^{I_1}u_1\cdot Z^{I_2}u_2\cdots Z^{I_m}u_m.
\end{equation}
Furthermore
\begin{equation}\label{eq 2 Leibniz}
\aligned
\del^IL^J(u_1\cdot u_2\cdots u_m\big) =& \sum_{I_1+I_2+\cdots I_m=I}\sum_{J_1+J_2\cdots J_m=J}\del^{I_1}L^{J_1}u_1\cdot \del^{I_2}L^{J_2}u_2\cdots\del^{I_m}L^{J_m}u_m
\\
=&\sum_{I_1+I_2+\cdots I_m=I\atop J_1+J_2+\cdots J_m=J}\del^{I_1}L^{J_1}u_1\cdot \del^{I_2}L^{J_2}u_2\cdots\del^{I_m}L^{J_m}u_m.
\endaligned
\end{equation}
\end{lemma}

\begin{lemma}[F\`aa di Bruno's formula]\label{lem 1 faa}
Let $u$ be a function defined in $\Kcal$, sufficiently regular. Let $f$ be a $C^{\infty}$ function defined on an open interval $I$ of $\RR$ which contains the image of $u$. Then the following identity holds:
\begin{equation}\label{eq 1 faa}
Z^If(u) = \sum_{1\leq k\leq |I|}f^{(k)}(u)\sum_{I_1+I_2+\cdots I_k\seq I}Z^{I_1}uZ^{I_2}u\cdots Z^{I_k}u.
\end{equation}
\end{lemma}

\section{Homogeneous functions and null condition}
\subsection{Homogeneous functions}
We recall the following notion on homogeneous functions:
\begin{definition}
Let $u$ be a $C^{\infty}$ function defined in $\{t>|x|\}$, satisfying the following properties:
\begin{itemize}
\item For a $k\in\RR$,  $u(\lambda t,\lambda x) = \lambda^ku(t,x),\quad \forall \lambda>0$.
\\
\item $\del^Iu(1,x)$ is bounded by a constant $C$ determined by $|I|$ and $u$ for $|x|< 1$.
\end{itemize}
Then $u$ is said to be {\sl homogeneous of degree $k$}.
\end{definition}
The following properties are immediate:
\begin{proposition}\label{prop 1 homo}
Let $u,v$ be homogeneous of degree $k,l$ respectively. Then
\begin{itemize}
\item When $k=l$, $\alpha u + \beta v$ is homogeneous of degree $k$ where $\alpha$ and $\beta$ are constants.
\\
\item $uv$ is homogeneous of degree $k+l$.
\\
\item $\del^IL^J u$ is homogeneous of degree $k-|I|$.
\\
\item There is a positive constant determined by $I,J$ and $u$ such that the following inequality holds in $\Kcal$:
\begin{equation}\label{eq 1 homo}
|\del^IL^Ju|\leq Ct^{k-|I|}.
\end{equation}
\end{itemize}
\end{proposition}
\begin{proof}
Only the third deserves a proof. Remark that
$$
u(\lambda t,\lambda x) = \lambda^ku(t,x)
$$
and derive the above identity with respect to $t$ and $x^a$ respectively,
$$
\lambda\del_{\alpha}u(\lambda t, \lambda x) = \lambda^k\del_{\alpha}u(t,x)\quad \Rightarrow\quad \del_{\alpha}u(\lambda t,\lambda x) = \lambda^{k-1}\del_{\alpha}u(t,x).
$$
and
$$
L_au(\lambda t, \lambda x) = (\lambda x^a)\del_tu(\lambda t,\lambda x) + (\lambda t)\del_au(\lambda t,\lambda x) = \lambda^kL_au(t,x).
$$
That is, when derived with respect $L_a$, the degree of homogeneity does not change, while derived respect to $\del_{\alpha}$, the degree of homogeneity decreases by $1$. Thus by induction the desired property is established.
\end{proof}

\subsection{Analysis on $(s/t)$ and $(t-r)/t$}
In this subsection we give the bounds on $Z^I(s/t)$ and $Z^I\big((t-r)/t\big)$. These results are established in \cite{MH-1}.

\begin{lemma}\label{lem 1 s/t}
In the region $\Kcal$, the following decompositions hold:
\begin{equation}\label{eq 1 lem 1 s/t}
L^J(s/t) = \Lambda^J(s/t),\quad \del^I(s/t) = \sum_{k=1}^{|I|}\Lambda_k^I(s/t)^{1-2k}
\end{equation}
with $\Lambda^J$ homogeneous of degree zero, $\Lambda_k^I$ homogeneous of degree $-|I|$. Furthermore,
\begin{equation}\label{eq 2 lem 1 s/t}
\big|\del^IL^J(s/t)\big|\leq \left\{
\aligned
&C(s/t), \quad &&|I|=0,
\\
&Cs^{-1},\quad &&|I|>0
\endaligned
\right.
\end{equation}
with $C$ a constant determined by $I,J$.
\end{lemma}
\begin{proof}
The first decomposition in \eqref{eq 1 lem 1 s/t} is by induction. We just remark that
$$
L_a(s/t) = \frac{-x^a}{t}(s/t)
$$
where $(-x^a/t)$ is homogeneous of degree zero.

For the second decomposition of \eqref{eq 1 lem 1 s/t}, we recall the Fa\`a di Bruno's formula and take $u = s^2/t^2 = (1-r^2/t^2)$ and
$$
\aligned
f:\RR^+&\rightarrow \RR
\\
x&\rightarrow x^{1/2}.
\endaligned
$$
Then
$$
\del^I(s/t) = \sum_{k=1}^{|I|}\sum_{I_1+\cdots+I_k\seq I}C_ku^{-k+1/2}\cdot \del^{I_1}u\del^{I_2}u\cdots \del^{I_k}u.
$$
Also recall that $(1-r^2/t^2)$ is homogeneous of degree zero,
$$
\del^{I_1}u \del^{I_2}u \cdots \del^{I_k}u \quad \text{ is homogeneous of degree }-|I|.
$$
So the desired decomposition is established.

Furthermore, recall proposition \ref{prop 1 homo} (the last point) and the fact that in $\Kcal$, $s\leq t\leq s^2$,
$$
\del^I(s/t)\leq C\sum_{k=1}^{|I|}(s/t)^{1-2k}t^{-|I|}\leq Cs^{-1}(t/s^2)^{|I|-1}\leq Cs^{-1}.
$$
Then by \eqref{eq 1 lem 1 s/t},
$$
\del^IL^J(s/t) = \del^I\big(\Lambda^J(s/t)\big) = \sum_{I_1+I_2=I}\del^{I_1}L^{J_1}\Lambda^J\cdot \del^{I_2}L^{J_2}(s/t).
$$
Recall the homogeneity of $\Lambda^{J}$, \eqref{eq 2 lem 1 s/t} is proved.
\end{proof}

Then we prove the following results:
\begin{lemma}\label{lem 3 s/t}
In the region $\Kcal$, the following bounds hold for $k,l\in\mathbb{Z}$:
\begin{equation}\label{eq 1 lem 3 s/t}
\big|\del^IL^J\big((s/t)^kt^l\big)\big|\leq
\left\{
\aligned
&C(s/t)^kt^l,\quad &&|I|=0,
\\
&C(s/t)^kt^l(t/s^2),\quad &&|I|\geq 1.
\endaligned
\right.
\end{equation}
\end{lemma}
\begin{proof}
We first establish the following bound, for $n\in\mathbb{Z}$:
\begin{equation}\label{eq 2' lem 1 s/t}
\big|\del^IL^J\big((s/t)^n\big)\big|\leq \left\{
\aligned
&C(s/t)^n, \quad &&|I|=0,
\\
&C(s/t)^n(t/s^2),\quad &&|I|\geq 0.
\endaligned
\right.
\end{equation}
When $n\in \mathbb{N}$, this is based on \eqref{eq 2 lem 1 s/t}. By Leibniz rule,
$$
\del^IL^J((s/t)^n) = \sum_{I_1+I_2+\cdots I_n=I\atop J_1+J_2+\cdots J_n=J}\del^{I_1}L^{J_1}(s/t)\cdot \del^{I_2}L^{J_2}(s/t)\cdots \del^{I_n}L^{J_n}(s/t).
$$
Remark that when $|I|\geq 1$, there are at least one $|I_j|\geq 1$.

Then consider $(s/t)^{-n}$. This is also by Fa\`a di Bruno's formula. We denote by $u = (s/t)$ and
$$
\aligned
f: \RR^+ &\rightarrow \RR
\\
 x&\rightarrow x^{-n}
\endaligned
$$
We denote by $Z^{I'} = \del^IL^J$. Then $Z^{I'}$ is of type $(j,i,0,0)$ with $i=|I|$ and $j = |J|$. Then
$$
\del^IL^J\big((s/t)^{-n}\big) = Z^{I'}(f(u)) =
\sum_{k=1}^{|I|+|J|}\sum_{I'_1+\cdots I'_k \seq I'}f^{(k)}(u)\cdot Z^{I_1'}(s/t)\cdots Z^{I_k'}(s/t).
$$
Here
$$
Z^{I'_l} = \del^{I_l}L^{J_l},\quad 1\leq l\leq k.
$$
Then by \eqref{eq 2 lem 1 s/t}: suppose that among $\{I_1,I_2\cdots I_k\}$ there are $i_0$ indices of order positive. Then when $i\geq 1$, there are  at least one index with order $\geq 1$. Then
$$
\big|f^{(k)}(u)\cdot\del^{I_1}L^{J_1}(s/t)\cdots \del^{I_k}L^{J_k}(s/t)\big|\leq C_n(s/t)^{-n-k}\cdot (s/t)^{k-i_0}s^{-i_0} = C(s/t)^{-n-i_0}s^{-i_0}.
$$
Recall that $s^{-1}\leq s/t$, then the bound on $\del^IL^J\big((s/t)^{-n}\big)$ is established.

Now for \eqref{eq 1 lem 3 s/t}, remark that
$$
\del^IL^J\big((s/t)^kt^l\big) = \sum_{I_1+I_2=I\atop J_1+J_2=J}\del^{I_1}L^{J_1}(s/t)^k\cdot \del^{I_2}L^{J_2}t^l.
$$
Then apply \eqref{eq 2' lem 1 s/t} and the homogeneity of $t^l$, the desired result is established.
\end{proof}

\begin{remark}
We list out some special cases of \eqref{eq 1 lem 3 s/t}:
\begin{equation}\label{eq 2 lem 2 s/t}
\big|\del^IL^J(s^n)\big|\leq
\left\{
\aligned
&Cs^n,\quad |I|=0,
\\
&Cs^n(t/s^2),\quad |I|\geq 1,
\endaligned
\right.
\quad
\big|\del^IL^J(s^{-n})\big|\leq
\left\{
\aligned
&Cs^{-n},\quad |I|=0,
\\
&Cs^{-n}(t/s^2),\quad |I|\geq 1.
\endaligned
\right.
\end{equation}
\end{remark}

\begin{lemma}\label{lem 1 (t-r)/t}
The following bounds hold with a constant $C$ determined by $I,J$:
in the region $\Kcal\cap\{t/2<|x|<t\}$,
\begin{equation}\label{eq 1 lem 1 (t-r)/t}
\big|\del^IL^J(1-r/t)\big|\leq \left\{
\aligned
&C(s/t)^2,\quad &&|I|=0,
\\
&Ct^{-|I|},\quad &&|I|>0.
\endaligned
\right.
\end{equation}
\end{lemma}
\begin{proof}
We first remark that
$$
L_a\big((r/t)^2\big) = \frac{2x^a}{t}(s/t)^2
$$
which leads to (by induction and \eqref{eq 1 lem 1 s/t}, we omit the detail)
\begin{equation}\label{eq 2 lem 1 (t-r)/t}
L^J ((r/t)^2) = \Theta^J(s/t)^2,\quad |J|\geq 1, \quad \Theta^J\text{ are homogeneous functions of degree zero.}
\end{equation}
Then by the homogeneity of $\Theta^J$ and $(s/t)^2$,
\begin{equation}\label{eq 3 lem 1 (t-r)/t}
\big|\del^IL^J\big((r/t)^2\big)\big|\leq
\left\{
\aligned
&C(s/t)^2,\quad &&|I|=0,|J|\geq 1,
\\
&Ct^{-|I|}\leq C(s/t)^2,\quad &&|I|\geq 1.
\endaligned
\right.
\end{equation}
Now we denote by $u = (r/t)^2$ which is homogeneous of degree zero, and
$$
\aligned
f:(0,1)&\rightarrow\RR^+
\\
x&\rightarrow \sqrt{x}.
\endaligned
$$
Then by denoting $Z^{I'} = \del^IL^J$, which is of type $(j,i,0,0)$ with $i=|I|$ and $j=|J|$. Then
$$
\del^IL^J(r/t) = Z^{I'}\big(f(u)\big) = \sum_{k=1}^{|I|+|J|}\sum_{I_1'+\cdots I_k'\seq I'}f^{(k)}(u)\cdot Z^{I_1'}u\cdots Z^{I_k'}u.
$$
Recall that
$$
f^{(k)}(u) = C_ku^{1/2-k} = C_k(r/t)^{1-2k}.
$$
Recalling that in the region $\{t/2<|x|<t\}$, $1/2<r/t<1$. Then  by \eqref{eq 3 lem 1 (t-r)/t},
$$
\big|\del^IL^J(r/t)\big| \leq  \left\{
\aligned
&C(s/t)^2,\quad &&|I|=0,
\\
&Ct^{-|I|},\quad &&|I|\geq 1.
\endaligned
\right.
$$
Finally, remark that $|1-(r/t)| \leq C (s/t)^2$. Then \eqref{eq 1 lem 1 (t-r)/t} is established.
\end{proof}

\subsection{Null conditions in SHF and HF}
Let $T$ and $Q$ be two$-$ and three$-$contravariant tensor fields respectively defined on $\Kcal$. Suppose that in $\Kcal$,
\begin{equation}\label{eq 1 null}
T^{\alpha\beta},\  Q^{\alpha\beta\gamma}\quad \text{are  constants}.
\end{equation}
Then the following bounds hold:
\begin{lemma}\label{lem 1 null}
For all $I,J$, in $\Kcal$ the following bounds hold
\begin{equation}\label{eq 1 lem 1 null}
\big|\del^IL^J\Qu^{\alpha\beta\gamma}\big| + \big|\del^IL^J \Tu^{\alpha\beta}\big|\leq Ct^{-|I|}
\end{equation}
and
\begin{equation}\label{eq 2 lem 1 null}
\big|(s/t)^k\del^IL^J\Qb^{\alpha\beta\gamma}\big| + \big|(s/t)^j\del^IL^J \Tb^{\alpha'\beta'}\big|\leq
\left\{
\aligned
&C,\quad |I|=0,
\\
&Ct/s^2,\quad |I|\geq 1.
\endaligned
\right.
\end{equation}
where $k,j$ are the number of $0$ contained in $(\alpha,\beta,\gamma)$ and $(\alpha',\beta')$ respectively.
\end{lemma}
\begin{remark}
Let us explain the last phrase in the statement of the above lemma by examples. For $\Qb^{001}$, $k=2$ and for $\Tb^{11}$, $j=0$.
\end{remark}
\begin{proof}
The proof is a direct calculation. Recall that
$$
\Tu^{\alpha\beta} = \Psiu^{\alpha}_{\alpha'}\Psiu^{\beta}_{\beta'}T^{\alpha'\beta'},\quad
\Tb^{\alpha\beta} = \Psib^{\alpha}_{\alpha'}\Psib^{\beta}_{\beta'}T^{\alpha'\beta'}.
$$

Then we make the following observation. $\Psiu^{\alpha}_{\alpha'}$ are homogeneous of degree zero. Thus
$$
\del^IL^J\Tu^{\alpha\beta} = \del^IL^J\big(\Psiu^{\alpha}_{\alpha'}\Psiu^{\beta}_{\beta'}T^{\alpha'\beta'}\big)
$$
is homogeneous of degree $-|I|$. Then the desired bound is established. The bounds on $\Qb^{\alpha\beta\gamma}$ is proved in the same manner. We omit the detail.

For the bounds on $\Tb^{\alpha\beta}$, we make the following observation.
$$
(s/t)\Psib_{\alpha}^0, \quad \Psib_{\alpha}^a
$$
are homogeneous of degree zero. Thus
$$
(s/t)^k\Tb^{\alpha\beta} = (s/t)^k\Psib_{\alpha'}^{\alpha}\Psib_{\beta'}^{\beta}T^{\alpha'\beta'}
$$
is homogeneous of degree zero. We denote by
$$
f := (s/t)^k\Tb^{\alpha\beta}
$$
Then
$$
(s/t)^k\del^IL^J(\Tb^{\alpha\beta}) = (s/t)^k\del^IL^J\big((s/t)^{-k}f\big) = (s/t)^k\sum_{I_1+I_2=I\atop J_1+J_2=J}\del^{I_1}L^{J_1}(s/t)^{-k}\cdot \del^{I_2}L^{J_2}f.
$$
Then by \eqref{eq 2' lem 1 s/t} and the homogeneity of $f$, the desired bound on $\Tb^{\alpha\beta}$ is established.

The bound on $\Qb^{\alpha\beta\gamma}$ is established in the same manner, we omit the detail.
\end{proof}

Then we recall the notion of null condition. $T$ and $Q$ are called null forms, if
$$
T^{\alpha\beta}\xi_{\alpha}\xi_{\beta} = Q^{\alpha\beta\gamma}\xi_{\alpha}\xi_{\beta}\xi_{\gamma} = 0,\quad \forall \xi\in\RR^{n+1}, \quad \xi_0^2-\sum_{a=1}^n\xi_a^2 = 0.
$$
The following bounds on null forms  are established and applied in diverse of context, see in detail \cite{LeFloch-MA-book-2015,Lindblad2008,LM2,LM3,MA2017-1,MA2017-2,MA2017-3}.
\begin{proposition}[Null condition in SHF and HF]\label{prop 1 null}
Let $T$, $Q$ be null forms of two and three contravariant type respectively. Suppose that in $\Kcal$,
$$
T^{\alpha\beta}, Q^{\alpha\beta\gamma},\quad \text{are constants}.
$$
Then
\begin{equation}\label{eq-SHF null}
|\del^IL^J\Qu^{000}| + |\del^IL^J \Tu^{00}|
\leq
\left\{
\aligned
&C(s/t)^2,\quad &&|I|=0,
\\
&Ct^{-|I|},\quad&&|I|>0.
\endaligned
\right.
\end{equation}
and
\begin{equation}\label{eq-HF null}
|(s/t)\del^IL^J\Qb^{000}|+|\del^IL^J\Tb^{00}|\leq
\left\{
\aligned
&C,\quad &&|I|=0,
\\
&Ct/s^2,\quad &&|I|>0.
\endaligned
\right.
\end{equation}
\end{proposition}
\begin{proof}
We remark the following identity:
$$
\Tu^{00} = \Psiu_{\alpha}^0\Psiu_{\beta}^0T^{\alpha\beta}.
$$
We denote by $\xi = (r/t,x^a/t)^T$. Remark that
$$
\xi_0^2 - \sum_a\xi_a^2 = 0.
$$
Then by null condition
$$
T^{\alpha\beta}\xi_{\alpha}\xi_{\beta} = 0.
$$
So
$$
\aligned
\Tu^{00} =& \Psiu_{\alpha}^0\Psiu_{\beta}^0T^{\alpha\beta} - \Psiu_{\alpha}^0\xi_{\beta}T^{\alpha\beta} + \Psiu_{\alpha}^0 \xi_{\beta}T^{\alpha\beta} - \xi_{\alpha}\xi_{\beta}T^{\alpha\beta}
\\
=& \Psiu_{\alpha}^0\big(\Psiu_{\beta}^0-\xi_{\beta}\big)T^{\alpha\beta} + (\Psiu_{\alpha}^0 - \xi_{\alpha})\xi_{\beta}T^{\alpha\beta}.
\endaligned
$$
Recall that
$$
\Psiu_{\alpha}^0 - \xi_{\alpha} = \left\{
\aligned
&1-r/t,\quad &&\alpha = 0,
\\
&0,\quad &&\alpha>0.
\endaligned
\right.
$$
Then
$$
\Tu^{00}  = (1-r/t)f
$$
with $f$ a homogeneous function. Then
$$
\del^IL^J (\Tu^{00}) = \sum_{I_1+I_2=I\atop J_1+J_2=J}\del^{I_1}L^{J_1}(1-r/t)\cdot \del^{I_2}L^{J_2}f
$$
Then by \eqref{eq 1 lem 1 (t-r)/t}, the bound on $\Tu^{00}$ is established in  $\{t/2<r<t\}$. By \eqref{eq 1 lem 1 null} and the fact that $\sqrt{3}/2\leq s/t\leq 1$ for $\{0\leq r\leq t/2\}$, this bound holds in $\Kcal$. 

The bound on $\Qu^{000}$ is established in the same way and we omit the detail.

For the bounds on the components in HF. We remark that
$$
\Tb^{00} = \Psib_{\alpha}^0\Psib_{\beta}^0T^{\alpha\beta} = (t/s)^2\Psiu_{\alpha}^0\Psiu_{\beta}^0T^{\alpha\beta} = (t/s)^2\Tu^{00}.
$$
Then
$$
\del^IL^J \Tb^{00} = \del^IL^J\big((t/s)^2\Tu^{00}\big) = \sum_{I_1+I_2=I\atop J_1+J_2=J}\del^{I_1}L^{J_1}(t/s)^2\cdot \del^{I_2}L^{J_2}\Tu^{00}
$$
Then by \eqref{eq-SHF null} and \eqref{eq 2' lem 1 s/t}, the bound on $\Tb^{00}$ is established.

The bound on $\Qb^{000}$ can be proved in the same manner, we omit the detail.
\end{proof}

\section{Decomposition of commutators}
\subsection{Basic decomposition}
We recall the following basic relations of commutation, established in \cite{LeFloch-MA-book-2015}:
\begin{lemma}\label{lem 1 decompo commu}
Let $u$ be a function defined in $\Kcal$, sufficiently regular. Let $(I,J)$ be a pair of multi-indices, then the following relations hold:
\begin{equation}\label{eq 1 comm}
[L^J,\del_{\alpha}] = \sum_{\beta,|J'|<|J|}\Gamma_{\alpha J'}^{J\,\beta}\del_{\beta}L^{J'},
\end{equation}
\begin{equation}\label{eq 2 comm}
[L^J,\del^I] = \sum_{|I'|=|I|\atop|J'|<|J|}\Gamma_{I'J'}^{JI}\del^{I'}L^{J'}
\end{equation}
where $\Gamma_{\alpha J'}^{J\,\beta}$ and $\Gamma_{I'J'}^{JI}$ are constants.
\end{lemma}
\begin{proof}
This is firstly proved in \cite{LeFloch-MA-book-2015}. Here we give a more detailed proof.

Firstly,
$$
[L_a,\del_t] = -\del_a,\quad [L_a,\del_b] = -\delta_{ab}\del_t.
$$
Then we denote by
\begin{equation}\label{eq 1 pr lem 1 decompo commu}
[L_a,\del_{\alpha}] = \gamma_{a\alpha}^{\beta}\del_{\beta}.
\end{equation}
with $\gamma_{\alpha a}^{\beta}$ constants.

Then \eqref{eq 1 comm} is by induction on $|J|$. The case $|J|=1$ is guaranteed by \eqref{eq 1 pr lem 1 decompo commu}. For $|J|\geq 1$, remark the following calculation (by applying the assumption of induction):
$$
\aligned
\,[L_aL^J,\del_{\alpha}] =& L_a\big([L^J,\del_{\alpha}]\big) + [L_a,\del_{\alpha}]L^J
\\
=&\sum_{|J'|<|J|}\Gamma_{\alpha J'}^{J\beta}L_a\del_{\beta}L^{J'} + \gamma_{a\alpha}^{\beta}\del_{\beta}L^J
\\
=&\sum_{|J'|<|J|}\Gamma_{\alpha J'}^{J\beta}\del_{\beta}L_a L^{J'} + \sum_{|J'|<|J|}\Gamma_{\alpha J'}^{J\beta}[L_a,\del_{\beta}]L^{J'}
+ \gamma_{a\alpha}^{\beta}\del_{\beta}L^J
\\
=&\sum_{|J'|<|J|}\Gamma_{\alpha J'}^{J\beta}\del_{\beta}L_a L^{J'} + \sum_{|J'|<|J|}\Gamma_{\alpha J'}^{J\beta}\gamma_{a\beta}^{\beta'}\del_{\beta'}L^{J'}
+ \gamma_{a\alpha}^{\beta}\del_{\beta}L^J
\endaligned
$$
and this guarantees the case with $|J''|=|J|+1$.

\eqref{eq 2 comm} is by induction on $|I|$. The case $|I|=1$ is guaranteed by \eqref{eq 1 comm}. Then we remark the following calculation (with assumption of induction):
$$
\aligned
\,[L^J,\del_{\alpha}\del^I] =&  [L^J,\del_{\alpha}]\del^I + \del_{\alpha}\big([L^J,\del^I]\big)
\\
=&\sum_{|J'|<|J|}\Gamma_{\alpha J'}^{J\beta}\del_{\beta}L^{J'}\del^I + \sum_{|I'|=|I|\atop |J'|<|J|}\Gamma_{I'J'}^{JI}\del_{\alpha}\del^{I'}L^{J'}
\\
=&\sum_{|J'|<|J|}\Gamma_{\alpha J'}^{J\beta}\del_{\beta}\del^I L^{J'} + \sum_{ |J'|<|J|}\Gamma_{\alpha J'}^{J\beta}\del_{\beta}[L^{J'},\del^I]
+ \sum_{|I'|=|I|\atop |J'|<|J|}\Gamma_{I'J'}^{JI}\del_{\alpha}\del^{I'}L^{J'}
\\
=&\sum_{|J'|<|J|}\Gamma_{\alpha J'}^{J\beta}\del_{\beta}\del^I L^{J'} + \sum_{|J'|<|J|,|I'|=|I|\atop|J''|<|J'|}\Gamma_{\alpha J'}^{J\beta}\Gamma^{J'I}_{I'J''}\del_{\beta}\del^{I'}L^{J''}
+ \sum_{|I'|=|I|\atop |J'|<|J|}\Gamma_{I'J'}^{JI}\del_{\alpha}\del^{I'}L^{J'}
\endaligned
$$
and by induction \eqref{eq 2 comm} is concluded.
\end{proof}

\subsection{Decomposition of high-order derivatives}


Before prove this we first establish a special case:
\begin{lemma}\label{lem 1 high-order}
Let $u$ be a function defined in $\Kcal_{[s_0,s_1]}$, sufficiently regular. Let $Z^K$ be a $N-$order operator of type $(j,i,0,0)$. Then the following bound holds:
\begin{equation}\label{eq 1 lem 1 high-order}
Z^Ku = \sum_{|I|=i\atop|J|\leq j}\Theta^K_{IJ}\del^IL^Ju
\end{equation}
with $\Theta_{IJ}^K$ constants determined by $K$ and $I,J$.
\end{lemma}
\begin{proof}[Proof of lemma \ref{lem 1 high-order}]
In this case, $Z^K$ can be written as
$$
Z^K = \del^{I_0}L^{J_0}\del^{I_1}L^{J_1}\cdots \del^{I_r}L^{J_r}.
$$
where $I_k$ and $J_k$ are multi-indices with components taking value in $\{n+1,n+2,\cdots 2n+1\}$ and $\{1,2,\cdots, n\}$ respectively. $I_0$ and $J_r$ may be empty indices (i.e. $L^{J_r}$ and $\del^{I_1}$ may be equal to $1$).

The proof is an induction on $r$. When $r=0$, that is,
$$
Z^K = \del^IL^J
$$
So \eqref{eq 1 lem 1 high-order} is trivial. Now suppose that for $r\leq r_0$ \eqref{eq 1 lem 1 high-order} holds, then we consider $r=r_0+1$:
\begin{equation}\label{eq 1 pr lem 1 high-order}
\aligned
Z^Ku =& \del^{I_0}L^{J_0}\del^{I_1}L^{J_1}\cdots \del^{I_r}L^{J_r}u
\\
=& \del^{I_0}\del^{I_1}L^{J_0}L^{J_1}\cdots \del^{I_r}L^{J_r}u + \del^{I_0}[L^{J_0},\del^{I_1}]\big(Z^{K''}u\big)
\endaligned
\end{equation}
where
$$
Z^{K''}u = L^{J_1}\del^{I_2}L^{J_2}\cdots \del^{I_{r}}L^{J_r}
$$
Denote by
$$
Z^{K'} = \del^{I_0}\del^{I_1}L^{J_0}L^{J_1}\cdots \del^{I_{r_0}}L^{J_{r_0}} = \del^{I'_0}L^{J'_0}\del^{I_2}L^{J_2}\cdots \del^{I_r}L^{J_r},
$$
so by the assumption of induction (remark that $Z^{K'}$ is also of type $(j,i,0,0)$),
$$
Z^{K'}u = \sum_{|I| = i\atop |J|\leq j}\Theta_{IJ}^{K'}\del^IL^Ju.
$$
For the second term in right-hand-side of \eqref{eq 1 pr lem 1 high-order}, we apply lemma \ref{lem 1 decompo commu}
$$
\aligned
\del^{I_0}[L^{J_0},\del^{I_1}]\big(Z^{K'}u\big)
=& \sum_{|I_1'|=|I_1|\atop |J_0'|<|J_0|}\del^{I_0}\Big(\Gamma_{J_0I_1}^{I_1'J_0'}\del^{I_1'}L^{J_0'}\big(Z^{K''}u\big)\Big)
\\
=&\sum_{|I_1'|=|I_1|\atop |J_0'|<|J_0|}\Gamma_{J_0I_1}^{I_1'J_0'}\big(\del^{I_0}\del^{I_1'}L^{J_0'}L^{J_1}\del^{I_2}L^{J_2}\cdots \del^{I_r}L^{J_r}u\big)
\\
=:&\sum_{|I_1'|=|I_1|\atop |J_0'|<|J_0|}\Gamma_{J_0I_1}^{I_1'J_0'}Z^{K'''}u.
\endaligned
$$
Then, apply the assumption of induction on $Z^{K'''}$ (remark that $Z^{K'''}$ is of type $(j',i,0,0)$ with $j'<j$)
$$
\del^{I_0}[L^{J_0},\del^{I_1}]\big(Z^{K'}u\big) = \sum_{|I_1'|=|I_1|\atop |J_0'|<|J_0|}\Gamma_{J_0I_1}^{I_1'J_0'}\sum_{|I|=i\atop |J|\leq j'<j }\Theta_{IJ}^{K'''}\del^IL^Ju.
$$
Recall that in the right-hand-side, the coefficients are constants, so \eqref{eq 1 lem 1 high-order} is established by induction.
\end{proof}

\begin{lemma}\label{lem 2 high-order}
Let $u$ be a function defined in $\Kcal_{[s_0,s_1]}$, sufficiently regular. Let $Z^K$ be a $N-$order operator of type $(j,i,0,l)$ and $N\geq 1$. Then the following bound holds:
\begin{equation}\label{eq 1 lem 2 high-order}
Z^Ku = \sum_{|I|\leq i, |J|\leq j+l\atop |I|+|J|\geq 1}t^{-l-i+|I|}\Delta_{IJ}^K\del^IL^Ju
\end{equation}
with $\Delta_{IJ}^K$ homogeneous functions of degree zero.
\end{lemma}
\begin{proof}
When $i=l=0$, $j=N$. Then $Z^K = L^K$ which is the form of \eqref{eq 1 lem 2 high-order}. When $i>0,l=0$, we apply \eqref{eq 1 lem 1 high-order}.

Then we proceed by induction on $l$. Suppose that \eqref{eq 1 lem 2 high-order} holds for $l\leq l_0$. Let $Z^K$ be of type $(j,i,0,l)$ with $l = l_0+1$. Suppose that $K = (k_1,k_2,\cdots k_m,\cdots ,k_N)$ with
$$
k_1,k_2,\cdots k_{m-1} \in \{1,2,\cdots 2n+1\}, \quad k_m\in\{3n+3,2n+3\cdots,4n+2\}.
$$
In another word, $Z_{k_m}$ is the first hyperbolic derivative in $Z^K$. We denote by $\delb_a = Z_{k_m}$. Then
$$
Z^Ku = Z^{K_1}\delb_aZ^{K_2}u
$$
with $Z^{K_1}$ being $(j_1,i_1,0,0)$ and $Z^{K_2}$ being $(j_2,i_2,0,l_0)$ with $j_1+j_2=j,i_1+i_2 = i$. Then
\begin{equation}\label{eq 1 pr lem 2 high-order}
Z^{K_1}\delb_aZ^{K_2}u = Z^{K_1}\big(t^{-1}L_aZ^{K_2}\big)u = \sum_{K_{11}+K_{12}=K_1}\!\!\!\!Z^{K_{11}}t^{-1}\cdot Z^{K_{12}}L_aZ^{K_2}u.
\end{equation}
Suppose that $K_{11}$ is of type $(j_{11},i_{11},0,0)$ and $K_{12}$ is of type $(j_{12},i_{12},0,0)$ with $i_{11}+i_{12} = i_1$ and $j_{11}+j_{12} = j_1$. Denote by $Z^{K_{11}'} = Z^{K_{12}}L_aZ^{K_2}$ and remark that $Z^{K_{11}'}$ is of type $(j_{11}',i_{11}',0,l_0)$ with
$$
i_{11}' = i_{12} + i_2,\quad j_{11}' = j_{12} + j_2 + 1.
$$
Then $i_{11}' + j_{11}' + l_0\geq 1$.  Then by the assumption of induction:
$$
Z^{K_{12}}L_aZ^{K_2}u = Z^{K_{11}'}u = \sum_{|I|\leq i'_{11},|J|\leq j_{11}' + 1 + l_0\atop |I|+|J|\geq 1}t^{-l_0 -i'_{11} + |I|}\Delta_{IJ}^{K_{11}'}\del^IL^Ju
$$

On the other hand, by the homogeneity of $t^{-1}$:
$$
|Z^{K_{11}}t^{-1}|\leq t^{-1-i_{11}}\theta
$$
where $\theta$ is a homogeneous function of degree zero. So for each term in right-hand-side of \eqref{eq 1 pr lem 2 high-order},
$$
\aligned
Z^{K_{11}}t^{-1}\cdot Z^{K_{12}}L_aZ^{K_2}u
=& \theta\sum_{|I|\leq i'_{11}|J|\leq j_{11}' + 1 + l_0\atop|I|+|J|\geq 1}\Delta_{IJ}^{K_{11}'}t^{-l_0-1-(i_{11}+i_{11}')+|I|}\del^IL^Ju
\\
=&\sum_{|I|\leq i'_{11} |J|\leq j_{11}' + l\atop |I|+|J|\geq 1}\theta\Delta_{IJ}^{K_{11}'}t^{-l-i+|I|}\del^IL^Ju
\endaligned
$$
and we remark that $\theta\Delta_{IJ}^{K_{11}'}$ are homogeneous functions of degree zero. Now we take the sum over $K_{11}+K_{12}=K_1$, and see that the case for $l = l_0+1$ is guaranteed (here remark that a sum of finite homogeneous functions of degree zero is again homogeneous of degree zero).
\end{proof}


\subsection{Estimates based on commutators}
For the convenience of discussion, we introduce the following integration on hyperboloids. Let $u$ be a function defined in $\Kcal_{[s_0,s_1]}$, sufficiently regular. Then
$$
\int_{\Hcal^*_s}u\,dx := \int_{\{r\leq (s^2-1)/2\}}u\big(\sqrt{s^2+r^2},x\big)\, dx
$$
and
$$
\|u\|^p_{L^p(\Hcal^*_s)} := \int_{\Hcal^*_s}|u|^p dx.
$$

We introduce the following hyperbolic energy:
$$
\aligned
E_{c}(s,u) =& \int_{\Hcal^*_s}\left(|\del_tu|^2 + 2(x^a/t)\del_tu\del_a u + \sum_{a=1}^n|\del_au|^2 + c^2u^2\right)\,dx
\\
=& \int_{\Hcal^*_s}\big(\sum_{a=1}^n|\delu_a u|^2 + \big|(s/t)\del_tu\big|^2 + c^2u^2\big)\, dx
\\
=& \int_{\Hcal^*_s}\big(\del_tu + (x^a/t)\del_au\big)^2 + \sum_{a=1}^n\big|(s/t)\del_au\big|^2 + c^2u^2\, dx
\endaligned
$$
Clearly, this energy controls the following quantities:
$$
\|(s/t)\del_{\alpha}u\|_{L^2(\Hcal^*_s)}, \quad \|\delu_a u\|_{L^2(\Hcal^*_s)},\quad \big\|(\del_t + (x^a/t)\del_a)u\big\|_{L^2(\Hcal^*_s)}
$$
and
$$
\|cu\|_{L^2(\Hcal^*_s)}.
$$

When $c=0$, we write $E(s,u) = E_0(s,u)$ for short. Furthermore:
$$
\Ecal_c^N(s,u):=\sum_{|I|+|J|\leq N}E_c(s,\del^IL^Ju).
$$
When $c=0$, we denote by $\Ecal_c^N(s,u) = \Ecal^N(s,u)$. Then we are ready to state the following result:
\begin{lemma}\label{lem 1 esti-high}
Let $u$ be a function defined in $\Kcal_{[s_0,s_1]}$, sufficiently regular. Let $Z^K$ be a operator of type $(j,i,0,l)$, and let $|K|=N+1\geq 1$. Then the following bounds hold:
\begin{equation}\label{eq 0 lem 1 esti-high}
 \|t^{l-1}Z^K u\|_{L^2(\Hcal^*_s)}\leq C\Ecal^N(s,u)^{1/2},\quad i=0,
\end{equation}
\begin{equation}\label{eq 1 lem 1 esti-high}
\|(s/t)t^l Z^Ku\|_{L^2(\Hcal^*_s)}\leq C\Ecal^N(s,u)^{1/2},\quad i\geq 1,
\end{equation}
When $c>0$, the following bound holds for $|K|\leq N-1$:
\begin{equation}\label{eq 3 lem 1 esti-high}
\|ct^l Z^K u\|_{L^2(\Hcal^*_s)}\leq C\Ecal^N_c(s,u)^{1/2},\quad
\end{equation}
\end{lemma}
\begin{proof}
\eqref{eq 0 lem 1 esti-high} is direct by \eqref{eq 1 lem 2 high-order}. To see this let us consider
$$
t^{-l-i+|I|}\Delta_{IJ}^K\del^IL^Ju,\quad |I|+|J|\geq 1.
$$

Recall that $|I|=i=0$, then $|J|\geq 1$. We denote by $L^J = L_aL^{J'}$. Then (recall $i\geq 0$)
$$
\aligned
\|t^{l-1}\big(t^{-l-i+|I|}\Delta_{IJ}^K\del^IL^Ju\big)\|_{L^2(\Hcal^*_s)} \leq& C \|t^{-1}L_aL^{J'}u\|_{L^2(\Hcal^*_s)}
\\
=& C\|\delu_aL^{J'}u\|_{L^2(\Hcal^*_s)}
\\
\leq& CE(s,L^{J'}u)^{1/2}\leq C\Ecal^{N-1}(s,u)^{1/2}.
\endaligned
$$


For \eqref{eq 1 lem 1 esti-high}, remark that in this case $i\geq 1$. By \eqref{eq 1 lem 2 high-order}, we consider
$$
t^{-l-i+|I|}\Delta_{IJ}^K\del^IL^Ju,\quad |I|+|J|\geq 1.
$$
As in discussion on \eqref{eq 0 lem 1 esti-high}, when $|I|\geq 1$, we denote by $\del^I = \del_{\alpha}\del^{I'}$. Then (recall that $i\geq|I|$)
$$
\aligned
\|t^l(s/t)\cdot t^{-l-i+|I|}\Delta_{IJ}^K\del^IL^Ju\|_{\Hcal^*_s}
\leq& \|t^{-i+|I|}(s/t)\del_{\alpha}\del^{I'}L^Ju\|_{\Hcal^*_s}
\\
\leq& CE(s,\del^{I'}L^Ju)^{1/2}\leq C\Ecal^{N-1}(s,u)^{1/2}.
\endaligned
$$

When $|I|=0$, then $|J|\geq 1$. We denote by $L^J = L_aL^{J'}$. Then (recall $i\geq 1$)
$$
\aligned
\|t^l(s/t)\big(t^{-l-i+|I|}\Delta_{IJ}^K\del^IL^Ju\big)\|_{L^2(\Hcal^*_s)} \leq& C \|t^{-i}L_aL^{J'}u\|_{L^2(\Hcal^*_s)}
\\
=& C\|t^{-i+1}\delu_aL^{J'}u\|_{L^2(\Hcal^*_s)}
\\
\leq& CE(s,L^{J'}u)^{1/2}\leq C\Ecal^{N-1}(s,u)^{1/2}.
\endaligned
$$

\eqref{eq 3 lem 1 esti-high} is direct by \eqref{eq 1 lem 2 high-order}, we omit the detail.
\end{proof}

The following result is to be combined Klainerman-Sobolev inequality in order to establish decay estimates. Let $p_n = [n/2]+1$.
\begin{lemma}\label{lem 2 esti-high}
Let $u$ be a function defined in $\Kcal_{[s_0,s_1]}$, sufficiently regular.  Let $|I_0|+|J_0|\leq p_n$, then the following bounds hold for $Z^K$ of type $(j,i,0,l)$ with $|K|\leq N-p_n+1$:
\begin{equation}\label{eq 0 lem 2 esti-high}
\big\|\del^{I_0}L^{J_0}\big(t^{l-1}Z^Ku\big)\big\|_{L^2(\Hcal^*_s)}\leq C\Ecal^N(s,u)^{1/2},\quad i=0
\end{equation}
\begin{equation}\label{eq 1 lem 2 esti-high}
\big\|\del^{I_0}L^{J_0}\big(t^l(s/t)Z^Ku\big)\big\|_{L^2(\Hcal^*_s)}\leq C\Ecal^N(s,t)^{1/2},\quad i\geq 1.
\end{equation}
When $c>0$:
\begin{equation}\label{eq 3 lem 2 esti-high}
\big\|c\del^{I_0}L^{J_0}\big(t^lZ^Ku\big)\big\|_{L^2(\Hcal^*_s)}\leq C\Ecal_c^N(s,t)^{1/2}.
\end{equation}
\end{lemma}
\begin{proof}
Recall \eqref{eq 1 lem 3 s/t} and the fact that $(t/s^2)\leq C$ in $\Kcal$. Then
$$
\del^{I_0}L^{J_0}\big(t^{l-1}Z^Ku\big) = \sum_{I_{01}+I_{02}=I_0\atop J_{01}+J_{02}=J_0}\del^{I_{01}}L^{J_{01}}t^{l-1}\cdot\del^{I_{02}}L^{J_{02}}Z^Ku
$$
Then each term in right-hand-side, we apply \eqref{eq 1 lem 3 s/t} on the first factor. For second factor, remark that
$$
\del^{I_{02}}L^{J_{02}}Z^K
$$
is of order $\leq N+1$. Then by \eqref{eq 1 lem 1 esti-high}, \eqref{eq 0 lem 2 esti-high} and \eqref{eq 1 lem 2 esti-high}  are established.

\eqref{eq 3 lem 2 esti-high} are established in the same manner, we omit the detail.
\end{proof}

\subsection{Estimates on Hessian form}
In this section, we concentrate on the estimates on the following terms:
$$
\del_{\alpha}\del_{\beta}Z^Ku,\quad Z^K\del_{\alpha}\del_{\beta}u.
$$
With a bit abuse of notation, we call these terms the {Hessian form} of $u$ of order $|K|$ . We first remark the following decomposition of Hessian form of a function with respect to SHF:
\begin{lemma}\label{lem 1 Hessian}
Let $u$ be a function defined in $\Kcal_{[s_0,s_1]}$, sufficiently regular. Then
\begin{equation}\label{eq 1 lem 1 Hessian}
\del_t\del_a u = -\frac{x^a}{t}\del_t\del_t u + t^{-1}\big(\del_tL_a - \delb_a + (x^a/t)\del_t\big)u
\end{equation}
\begin{equation}\label{eq 2 lem 1 Hessian}
\del_a\del_bu = \frac{x^ax^b}{t^2}\del_t\del_t + t^{-1}\big(\del_aL_b - (x^b/t)\del_tL_a + (x^b/t)\delb_a - \delta_{ab}\del_t - (x^ax^b/t^2)\del_t\big)
\end{equation}
\end{lemma}
\begin{proof}
These can be verified by direct calculation and we omit the detail.
\end{proof}

Then by the above lemma, we have the following decomposition of the D'Alembert operator with respect to SHF:
\begin{equation}\label{eq 1 decompo-box-semi}
\Box = (s/t)^2\del_t\del_t + t^{-1}\left((2x^a/t)\del_tL_a - \sum_a\delu_aL_a - (x^a/t)\delu_a + (n+(r/t)^2)\del_t\right)
\end{equation}
This can be verified by direct calculation, we omit the detail. Then
\begin{equation}\label{eq 2 decompo-box-semi}
(s/t)^2\del_t\del_tu = \Box u - t^{-1}A_m[u]
\end{equation}
with
\begin{equation}
A_m[u] = \left((2x^a/t)\del_tL_a - \sum_a\delu_aL_a - (x^a/t)\delu_a + (n+(r/t)^2)\del_t\right)u.
\end{equation}
Here $m$ represents the Minkowski metric. Now we are ready to establish the following result:
\begin{lemma}\label{lem 2 Hessian}
Let $u$ be a function defined in $\Kcal_{[s_0,s_1]}$, sufficiently regular. Suppose that $Z^K$ is of type $(j,i,0,0)$. Then the following bounds hold for $|K|\leq N-1$:
\begin{equation}\label{eq 1 lem 2 Hessian}
\big\|t(s/t)^3\del_{\alpha}\del_{\beta}Z^Ku\big\|_{L^2(\Hcal^*_s)} + \big\|t(s/t)^3Z^K\del_{\alpha}\del_{\beta}u\big\|_{\Hcal^*_s}\leq C\Ecal^N(s,u)^{1/2} + C\sum_{|K'|\leq K}\|sZ^{K'}\Box u\|_{L^2(\Hcal^*_s)}.
\end{equation}
Furthermore, if we take $|I_0|+|J_0|\leq p_n = [n/2]+1$, then for all $Z^K$ of type $(j,i,0,0)$ with $|K'|\leq N-p_n-1$,
\begin{equation}\label{eq 2 lem 2 Hessian}
\aligned
\big\|\del^{I_0}L^{J_0}\big(t(s/t)^3\del_{\alpha}\del_{\beta}Z^{K'}u\big)\big\|_{L^2(\Hcal^*_s)}
+& \big\|\del^{I_0}L^{J_0}\big(t(s/t)^3Z^{K'}\del_{\alpha}\del_{\beta}u\big)\big\|_{\Hcal^*_s}
\\
\leq& C\Ecal^N(s,u)^{1/2} + C\sum_{|K''|\leq N-1}\|sZ^{K''}\Box u\|_{L^2(\Hcal^*_s)}
\endaligned
\end{equation}
\end{lemma}
\begin{proof}
By \eqref{eq 2 decompo-box-semi},
$$
t(s/t)^3\del_t\del_tv = s\Box v - (s/t)A_m[v].
$$
Remark that
$$
\|(s/t)A_m[v]\|_{L^2(\Hcal^*_s)}\leq C\sum_{|J|\leq 1}E(s,L^Jv)^{1/2},
$$
Then
\begin{equation}\label{eq 1 pr lem 2 Hessian}
\|t(s/t)^3\del_t\del_tv\|_{L^2(\Hcal^*_s)}\leq \|s\Box v\|_{L^2(\Hcal^*_s)} + C\Ecal^1(s,v)^{1/2}
\end{equation}

Then recall \eqref{eq 1 lem 1 Hessian} and \eqref{eq 2 lem 1 Hessian}, we see that (remark that $x^a/t$ are homogeneous of degree zero, thus bounded in $\Kcal$):
$$
\|t(s/t)^3\del_t\del_av\|_{L^2(\Hcal^*_s)} + \|t(s/t)^3\del_a\del_bv\|_{L^2(\Hcal^*_s)}\leq C\|t(s/t)^3\del_t\del_tv\|_{L^2(\Hcal^*_s)} +  C\Ecal^1(s,v)^{1/2}.
$$
Combined with \eqref{eq 1 pr lem 2 Hessian},
\begin{equation}\label{eq 2 pr lem 2 Hessian}
\|t(s/t)^3\del_{\alpha}\del_{\beta}v\|_{L^2(\Hcal^*_s)}\leq C\|s\Box v\|_{L^2(\Hcal^*_s)} + C\Ecal^1(s,v)^{1/2}.
\end{equation}
Then, we take $v = Z^Ku$ with $Z^K$ being of type $(j,i,0,0)$ and $|K|\leq N-1$. Recall that
$$
[Z^K,\Box] = 0,\quad\text{ if $Z^K$ is composed by $\del_{\alpha}$ and $L_a$.}
$$
Furthermore, recall \eqref{eq 1 lem 1 high-order}, $Z^Ku$, $L_aZ^Ku$ are linear combinations of $\del^IL^Ju$ with $|I|+|J|\leq N-1$ with constant coefficients. So
$$
\Ecal^1(s,Z^Ku)^{1/2}\leq C\Ecal^N(s,u)^{1/2}.
$$
Then
\begin{equation}\label{eq 3 pr lem 2 Hessian}
\|t(s/t)^3\del_{\alpha}\del_{\beta}Z^Ku\|_{L^2(\Hcal^*_s)}\leq C\|sZ^K \Box u\|_{L^2(\Hcal^*_s)} + \Ecal^N(s,u)^{1/2}.
\end{equation}

Once \eqref{eq 3 pr lem 2 Hessian} is proved, we recall \eqref{eq 1 lem 1 high-order} and see that the following terms
$$
Z^K\del_{\alpha}\del_{\beta}u
$$
is a linear combination of $\del_{\alpha}\del_{\beta}\del^IL^Ju$ with $|I|+|J|\leq N-1$. Then by \eqref{eq 3 pr lem 2 Hessian}, \eqref{eq 1 lem 2 Hessian} is established.

For \eqref{eq 2 lem 2 Hessian}, we only need to recall \eqref{eq 1 lem 3 s/t} and the following calculation:
$$
\del^{I_0}L^{J_0}\big(t(s/t)^3\del_{\alpha}\del_{\beta}Z^{K'}u\big) =
\sum_{I_{01}+I_{02} = I_0\atop J_{01}+J_{02}=J_0}\del^{I_{01}}L^{J_{01}}\big(t(s/t)^3\big)\cdot \del^{I_{02}}L^{J_{02}}\del_{\alpha}\del_{\beta}Z^{K'}u.
$$
\end{proof}

\appendix
\section{Proof of lemma \ref{lem 1 Leinbiz} and \ref{lem 1 faa}}\label{sec A-1}
For the convenience of discussion, for $I = (i_1,i_2,\cdots i_N)$ we consider $I' = (i_0,i_1,\cdots i_N)$ with $i_0\in \{1,2,\cdots 4n+2\}$. The graph of $I'$ is
$$
\mathcal{G}(I') = \{(k,i_k)|k=0,1,2,\cdots N\}.
$$
Then we introduce the following notation
$$
\mathscr{D}_{ml}(I') = \big\{(\mathcal{I'}_1,\mathcal{I'}_2\cdots \mathcal{I'}_m)\in \mathscr{D}_m(I')\big|(0,i_0)\in\mathcal{I}'_l\big\}.
$$
Then
$$
\mathscr{D}_m(I') = \bigcup_{l=1}^m \mathscr{D}_{ml}(I'),\quad \mathscr{D}_{ml}(I') \cap \mathscr{D}_{ml'}(I') = \emptyset,\quad l\neq l'.
$$

We define the following map for $1\leq l\leq m$:
$$
\aligned
p_l :\quad  &\mathscr{D}_m(I) &&\rightarrow \mathscr{D}_{ml}(I')
\\
(\mathcal{I}_1,\mathcal{I}_2,\cdots &\mathcal{I}_l,\cdots,\mathcal{I}_m) &&\rightarrow (\mathcal{I}_1,\mathcal{I}_2,\cdots\mathcal{I'}_l,\cdots ,\mathcal{I}_m).
\endaligned
$$
with $\mathcal{I'}_l = \{(0,i_0)\}\cup \mathcal{I}_l$.

\begin{lemma}\label{lem 1 partition}
With the above notation, $p_l$  are bijective.
\end{lemma}
\begin{proof}
It is obvious that $p_l$ are injective. To prove that $p_l$ are surjective, let $(\mathcal{I'}_1,\mathcal{I'}_2,\cdots\mathcal{I'}_l,\cdots ,\mathcal{I'}_m)\in\mathscr{D}_{ml}(I')$ with $\mathcal{I'}_k\subset \mathcal{G}(I')$. Recall the definition of $\mathscr{D}_{ml}$,
$$
(0,i_0)\in \mathcal{I'}_l.
$$
This leads to the fact that when $k\neq l$,
$$
D(\mathcal{I'}_k)\subset \{1,2,\cdots N\} \quad \Rightarrow \quad \mathcal{I'}_k\subset \mathcal{G}(I)
$$
and
$$
D(\mathcal{I'}_l\backslash \{(0,i_0)\}) \subset\{1,2,\cdots N\} \quad \Rightarrow \quad \mathcal{I'}_l\backslash \{(0,i_0)\}\subset \mathcal{G}(I).
$$
Then we define
$$
\mathcal{I}_k = \left\{
\aligned
&\mathcal{I'}_k,\quad &&k\neq l,
\\
&\mathcal{I'}_l\backslash \{(0,i_0)\},\quad &&k=l.
\endaligned
\right.
$$

It is clear that for $k\neq k'$,  $\mathcal{I'}_k\cap\mathcal{I'}_{k'}=\emptyset\Rightarrow \mathcal{I}_k\cap\mathcal{I}_{k'}=\emptyset$ . Furthermore, from the fact
$$
\mathcal{G}(I') = \bigcup_{k=1}^m\mathcal{I'}_k
$$
we see that
$$
\mathcal{G}(I) = \mathcal{G}(I')\backslash\{(0,i_0)\} = \bigcup_{k=1}^m\mathcal{I}_k.
$$

Thus $(\mathcal{I}_1,\mathcal{I}_2,\cdots \mathcal{I}_m)\in\mathscr{D}_m(I)$ is the preimage of $(\mathcal{I'}_1,\mathcal{I'}_2,\cdots \mathcal{I'}_m)$ with respect to $p_l$.
\end{proof}

\begin{proof}[Proof of lemma \ref{lem 1 Leinbiz}]
This is by induction on $|I|$. When $|I|=1$, let $Z^I = Z_k\in \mathscr{Z}$. Remark that $Z_k$ is a first-order differential operator, then by Leibniz rule,
\begin{equation}\label{eq 1 pr lem Leibniz}
Z_k\big(u_1\cdot u_2\cdots u_m\big) = Z_ku_1\cdot u_2\cdots u_m + u_1\cdot Z_ku_2\cdots u_m + \cdots u_1\cdot u_2\cdots Z_ku_m.
\end{equation}
On the other hand, remark that for the set $\{(1,k)\}$, all possible $m-$partitions are in the following form:
$$
\{(1,k)\} = \bigcup_{l=1}^m \mathcal{I}_l,\quad
\mathcal{I}=\left\{
\aligned
&\{(1,k)\},\quad &&l = l_0,
\\
&\emptyset,\quad &&l\neq l_0.
\endaligned
\right.
$$
with $l_0 = 1,2,\cdots m$. Thus in the right-hand-side of \eqref{eq 1 pr lem Leibniz} the sum runs over all possible $1-$partituion of $Z_k$ and this concludes the case of \eqref{eq 1 Leibniz} with $|I|=1$.

Now let $Z^{I'} = Z_{i_0}Z^I$ with $Z_{i_0}\in\mathscr{Z}$ and we denote by $I' = (i_0,i_1,\cdots i_N)$ with graph $\mathcal{G}(I') = \{(k,i_k)|k=0,1,2,\cdots N\}$. Then
\begin{equation}\label{eq 2 pr lem Leibniz}
\aligned
&Z_{i_0}Z^I(u_1\cdot u_2\cdots u_m)
\\
=& Z_{i_0}\bigg(\sum_{I_1+I_2+\cdots +I_m = I}Z^{I_1}u_1\cdot Z^{I_2}u_2\cdots Z^{I_m}u_m\bigg)
\\
=& \sum_{I_1+I_2+\cdots +I_m = I}Z_{i_0}Z^{I_1}u_1\cdot Z^{I_2}u_2\cdots Z^{I_m}u_m
 + \sum_{I_1+I_2+\cdots +I_m = I}Z^{I_1}u_1\cdot Z_{i_0}Z^{I_2}u_2\cdots Z^{I_m}u_m
\\
 &+\cdots
 + \sum_{I_1+I_2+\cdots +I_m = I}Z^{I_1}u_1\cdot Z^{I_2}u_2\cdots Z_{i_0}Z^{I_m}u_m
\\
=:& S_1 + S_2 + \cdots S_m.
\endaligned
\end{equation}
We will prove that
\begin{equation}\label{eq 3 pr lem Leibniz}
S_l = \sum_{\mathscr{D}_{ml}(I')}Z^{I'_1}u_1\cdot Z^{I'_2}u_2\cdots Z^{I'_m}u_m.
\end{equation}

Recall $S_l$ is a sum over $\mathscr{D}_m(I)$. We denote by $\mathcal{S}_l =: \{$the terms in $S_l\}$. Then there is a bijection $i_l$ from $\mathcal{S}_l$ to $\mathscr{D}_m$. Then $ p_l\circ i_l$ is a bijection from $\mathcal{S}_l$ to $\mathscr{D}_{ml}(I')$. Furthermore, for a term
\begin{equation}\label{eq 4 pr lem Leibniz}
Z^{I_1}u_1\cdot Z^{I_2}u_2\cdots Z_{i_0}Z^{I_l}u_l\cdots Z^{I_m}u_m\in \mathcal{S}_l
\end{equation}
its image under $p_l\circ i_l$ is
\begin{equation}\label{eq 4' pr lem Leibniz}
(\mathcal{I'}_1,\mathcal{I'}_2,\cdots, \mathcal{I'}_l,\cdots \mathcal{I'}_m).
\end{equation}
The realization of the above partition is by definition contained in right-hand-side of \eqref{eq 3 pr lem Leibniz}. Furthermore, the realization of each $\mathcal{I'}_k$ is $Z^{I_k}$ when $k\neq l$ or $Z_{i_0}Z^{I_l}$ when $k=l$. This leads to the fact that the realization of \eqref{eq 4' pr lem Leibniz} is exactly \eqref{eq 4 pr lem Leibniz}. This establishes an one to one correspondence from $\mathcal{S}_l$ to the terms in right-hand-side of \eqref{eq 3 pr lem Leibniz}.

Then, recall that $\mathscr{D}_{ml}(I')$ is a partition of $\mathscr{D}_m(I')$. So
$$
\sum_{I_1'+I_2'+\cdots +I_m'=I'}Z^{I_1'}u_1\cdot Z^{I_2'}u_2\cdots Z^{I_m'}u_m = \sum_{l=1}^m\sum_{\mathscr{D}_{ml}(I')} Z^{I_1'}u_1\cdot Z^{I_2'}u_2\cdots Z^{I_m'}u_m.
$$
And this together with \eqref{eq 3 pr lem Leibniz} leads to the desired result.

For \eqref{eq 2 Leibniz}, we just apply twice \eqref{eq 1 Leibniz}.
\end{proof}

For the proof of lemma \ref{lem 1 faa}, we need the following observation.
We denote by
$$
\mathscr{D}_{m0}^*(I') = \big\{\{\mathcal{I}_1,\mathcal{I}_2,\cdots,\mathcal{I}_m\}\in\mathscr{D}^*_m(I')\big|\mathcal{I}_1 = \{(0,i_0)\}\big\}
$$
and
$$
\mathscr{D}_m^*(I')^+ = \mathscr{D}_m^*(I')\backslash \mathscr{D}_{m0}^*(I').
$$

Then we define
$$
\mathscr{D}_{ml}^*(I') := \big\{\{\mathcal{I}_1,\mathcal{I}_2,\cdots,\mathcal{I}_m\}\in\mathscr{D}^*_m(I')^+\big| (0,i_0)\in\mathcal{I}_l \big\}.
$$

Now as the case of $m-$ decomposition, we establish the following bijections:
$$
\aligned
p_0:\ \mathscr{D}_{m-1}^*(I) &\rightarrow \mathscr{D}_{m0}^*(I')
\\
\{\mathcal{I}_1,\cdots \mathcal{I}_{m-1}\}&\rightarrow\{(0,i_0),\mathcal{I}_1,\cdots ,\mathcal{I}_{m-1}\}
\endaligned
$$
and for $m\geq l\geq 1$,
$$
\aligned
p_l:\  \mathscr{D}_m^*(I) &\rightarrow\mathscr{D}_{ml}^*(I')
\\
\{\mathcal{I}_1,\mathcal{I}_2\cdots \mathcal{I}_m\}&\rightarrow\{\mathcal{I'}_1,\mathcal{I'}_2\cdots ,\mathcal{I'}_m\}
\endaligned
$$
with
$$
\mathcal{I'}_k = \mathcal{I}_k \quad\text{if } k\neq l,\quad  \mathcal{I'}_l = \mathcal{I}_l\cup \{(0,i_0)\}.
$$
Then we have the following results:
\begin{lemma}\label{lem 2 partition}
With the above notation, $p_l^*$ are bijective for $0\leq l\leq m$.
\end{lemma}
\begin{proof}
It is clear that $p_l^*$ are injective. To prove that they are surjective, we make the following observation. When $l=0$, for any partition in $\mathscr{D}_{m0}^*(I')$, by definition we see that it is of the following form:
$$
\{(0,i_0),\mathcal{I}_1,\cdots ,\mathcal{I}_{m-1}\}
$$
and $\mathcal{I}_j\subset \mathcal{G}(I)$ for $j=1,2,\cdots, m-1$. Also, $\mathcal{I}_j\cap\mathcal{I}_k = \emptyset$ when $j\neq k$. Thus $\{\mathcal{I}_1,\cdots ,\mathcal{I}_{m-1}\}\in\mathscr{D}_{m-1}^*(I)$ is the $p_0^*$ preimage of $\{(0,i_0),\mathcal{I}_1,\cdots ,\mathcal{I}_{m-1}\}$ which proves that $p_0^*$ is surjective.

For $1\leq l\leq m$. Let $\{\mathcal{I'}_1,\cdots ,\mathcal{I'}_m\}\in\mathscr{D}_{ml}^*(I')$. By definition
$$
(0,i_0)\in\mathcal{I'}_l.
$$
Then we define
\begin{equation}\label{eq 1 pr lem 2 partition}
(\mathcal{I}_1,\mathcal{I}_2,\cdots \mathcal{I}_k)
\end{equation}
with $\mathcal{I}_j = \mathcal{I'}_j$ when $j\neq l$ and $\mathcal{I}_l = \mathcal{I'}_l\backslash \{(0,i_0)\}$. Also, $D(\mathcal{I'}_j)\subset\{1,2,\cdots N\}$. Then \eqref{eq 1 pr lem 2 partition} is the contained in $\mathscr{D}_{m}^*(I')$ which is the preimage of $\{\mathcal{I'}_1,\cdots ,\mathcal{I'}_m\}$. This proves that $p_l^*$ are surjective.
\end{proof}

\begin{proof}[Proof of lemma \ref{lem 1 faa}]
This is checked by induction on $|I|$. For $|I|=1$ this is the chain rule. Suppose that \eqref{eq 1 faa} is valid for  $|I|\leq N$. For $Z_{i_0}\in\mathscr{Z}$, we consider
$$
\aligned
Z_{i_0}Z^If(u) =& Z_{i_0}\left(\sum_{1\leq k\leq |I|}f^{(k)}(u)\sum_{I_1+I_2+\cdots I_k\seq I}Z^{I_1}uZ^{I_2}u\cdots Z^{I_k}u\right)
\\
=&\sum_{1\leq k\leq |I|}f^{(k+1)}(u)\sum_{I_1+I_2+\cdots I_k\seq I}Z_{i_0}u \cdot Z^{I_1}uZ^{I_2}u\cdots Z^{I_k}u
\\
&+ \sum_{1\leq k\leq |I|}f^{(k)}(u)\sum_{I_1+I_2+\cdots I_k\seq I}Z_{i_0}\big(Z^{I_1}uZ^{I_2}u\cdots Z^{I_k}u\big)
\\
=& f'(u)Z_{i_0}Z^Iu
\\
&+ \sum_{k=2}^{|I|+1}f^{(k)}(u)\!\!\!\!\!\sum_{I_1+I_2\cdots+I_{k-1}\seq I}\!\!\!\!\!Z_{i_0}u\cdot Z^{I_1}uZ^{I_2}u\cdots Z^{I_{k-1}}u
\\
&+ \sum_{k=2}^{|I|+1}f^{(k)}(u)\sum_{I_1+I_2\cdots+I_k\seq I} Z_{i_0}\big(Z^{I_1}uZ^{I_2}u\cdots Z^{I_k}u\big)
\\
=:&f'(u)\cdot L_1 + \sum_{k=2}^{|I|+1}f^{(k)}\cdot L_k.
\endaligned
$$
On the other hand we denote by
$$
R_k := \sum_{I'_1+I'_2+\cdots +I'_k \seq I'}Z^{I'_1}uZ^{I'_2}u\cdots Z^{I'_k}u
$$
Then
$$
\sum_{k=1}^{|I|'}f^{(k)}(u)Z^{I'_1}uZ^{I'_2}u\cdots Z^{I'_k}u = \sum_{k=1}^{|I|'}f^{(k)}(u)R_k.
$$
So we only need to prove
\begin{equation}\label{eq 1 pr lem 1 faa}
L_k = R_k,\quad k=1,2,\cdots, |I'| = |I|+1.
\end{equation}

When $k=1$,
$$
R_1 = Z^{I'}u = Z_{i_0}Z^Iu=L_1.
$$

When $k>1$, we decompose $\mathscr{D}_k^*(I')$ into $k+1$ disjoint subsets:
$$
\mathscr{D}_k^*(I') = \bigcup_{l=0}^k\mathscr{D}_{kl}^*(I').
$$
and we write $R_k$ in the following form:
$$
R_k = \sum_{l=0}^kR_{kl}
$$
with
$$
R_{kl} := \sum_{\mathscr{D}_{kl}^*(I')}\!\!\!\!\! Z^{I_1'}uZ^{I_2'}u\cdots Z^{I_k'}u.
$$

On the other hand, we write $L_k$ in to the following form:
$$
L_k = \sum_{l=0}^kL_{kl}
$$
with
$$
L_{k0} := \sum_{I_1+I_2+\cdots +I_{k-1}\seq I}Z_{i_0}u Z^{I_1}u\cdots Z^{I_{k-1}}u,
$$
and for $1\leq l\leq k$,
$$
L_{kl} := \sum_{I_1+I_2+\cdots +I_k \seq I}Z^{I_1}u\cdots Z_{i_0}Z^{I_l}u\cdots Z^{I_k}u.
$$
Then, remark that $L_{kl}$ is a sum over $\mathscr{D}_k^*(I)$ or $\mathscr{D}_{k-1}^*(I)$.

\eqref{eq 1 pr lem 1 faa} is guaranteed by
\begin{equation}\label{eq 2 pr lem 1 faa}
L_{kl} = R_{kl}.
\end{equation}
When  $l = 0$, recall the bijection $p_0^*$ from $\mathscr{D}^*_{k-1}(I)$ to $\mathscr{D}_{k0}^*(I')$. Let
\begin{equation}\label{eq 3 pr lem 1 faa}
(\mathcal{I}_1,\mathcal{I}_2,\cdots, \mathcal{I}_{k-1})\in\mathscr{D}^*_{k-1}(I)
\end{equation}
and we see
\begin{equation}\label{eq 4 pr lem 1 faa}
((0,i_0),\mathcal{I}_1,\mathcal{I}_2,\cdots, \mathcal{I}_{k-1}) = p_0^*(\mathcal{I}_1,\mathcal{I}_2,\cdots, \mathcal{I}_{k-1})
\end{equation}
We remark that the term in $L_{k0}$ corresponding to \eqref{eq 3 pr lem 1 faa} equals to the term in $R_{k0}$ corresponding to the $p_0^*$ image of \eqref{eq 3 pr lem 1 faa}. Then $L_{k0} = R_{k0}$.

In the same manner, let
\begin{equation}\label{eq 5 pr lem 1 faa}
(\mathcal{I}_1,\cdots, \mathcal{I}_k)\in \mathscr{D}_k^*(I)
\end{equation}
and denote by
\begin{equation}\label{eq 6 pr lem 1 faa}
(\mathcal{I'}_1,\cdots, \mathcal{I'}_k) = p_l^*(\mathcal{I}_1,\cdots, \mathcal{I}_k).
\end{equation}
Then we remark that the term in $L_{kl}$ corresponding to \eqref{eq 5 pr lem 1 faa} equals to the term in $R_{kl}$ corresponding to \eqref{eq 6 pr lem 1 faa}. Thus $L_{kl} = R_{kl}$.

Then we conclude by induction the desired result.
\end{proof}

\bibliographystyle{elsarticle-num}
\bibliography{RCbib}

\begin{thebibliography}{1}
\expandafter\ifx\csname url\endcsname\relax
  \def\url#1{\texttt{#1}}\fi
\expandafter\ifx\csname urlprefix\endcsname\relax\def\urlprefix{URL }\fi
\expandafter\ifx\csname href\endcsname\relax
  \def\href#1#2{#2} \def\path#1{#1}\fi

\bibitem{LeFloch-MA-book-2015}
P.~LeFloch, Y.~Ma, The hyperboloidal foliation method, World Scientific, 2015.

\bibitem{LM2}
P.~LeFloch, Y.~Ma, The nonlinear stability of {M}inkowski space for
  self-gravitating massive field. {T}he wave-{K}lein-{G}ordon model, Commun.
  Math. Phys. 346~(2)  603--665.
\newblock \href {http://dx.doi.org/10.1007/s00220-015-2549-8}
  {\path{doi:10.1007/s00220-015-2549-8}}.

\bibitem{LM3}
P.~LeFloch, Y.~Ma, The global nonlinear stability of {M}inkowski space for
  self-gravitating massive fields, World Scientific, 2017.

\bibitem{MA2017-1}
Y.~Ma, Global solutions of quasilinear wave-{K}lein-{G}ordon system in two
  space dimension: technical tools, J. Hyperbol. Differ. Eq. 14~(4)  591--625.
\newblock \href {http://dx.doi.org/10.1142/S0219891617500205}
  {\path{doi:10.1142/S0219891617500205}}.

\bibitem{MA2017-2}
Y.~Ma, Global solutions of quasilinear wave-{K}lein-{G}ordon system in two
  space dimension: completion of the proof, J. Hyperbol. Differ. Eq. 14~(4)
  627--670.
\newblock \href {http://dx.doi.org/10.1142/S0219891617500217}
  {\path{doi:10.1142/S0219891617500217}}.

\bibitem{MA2017-3}
Y.~Ma, Global solutions of non-linear wave-{K}lein-{G}ordon system in two space
  dimensions: semi-linear interactions, arXiv:1712.05315v1.

\bibitem{MH-1}
Y.~MA, H.-J. HUANG, A conformal-type energy inequality on hyperboloids and its
  application to quasi-linear wave equation in $\mathbb{R}^{3+1}$,
  arXiv:1711.00498v1.

\end{thebibliography}

\end{document}